\documentclass[10pt]{amsart}
\usepackage[latin1]{inputenc}
\usepackage{subfigure}
\usepackage{amsmath,amssymb}
\usepackage{enumitem}
\usepackage{units}
\usepackage{graphicx}
\usepackage{a4wide}

\newtheorem{definition}{Definition}[section] 
\newtheorem{proposition}[definition]{Proposition} 
\newtheorem{theorem}[definition]{Theorem} 
 
\newtheorem{lemma}[definition]{Lemma} 
\newtheorem{example}[definition]{Example} 
\newtheorem{corollary}[definition]{Corollary}

\newcommand{\lleft}{( \! (}
\newcommand{\rright}{) \! )}
\newcommand{\mik}{\preceq}

\newcommand{\lk}{\rm link}

\newcommand{\NC}{\mathrm{NC}}
\newcommand{\I}{\mathrm{I}}
\newcommand{\rank}{\mathrm{rank}}
\newcommand{\Abs}{\mathrm{Abs}}

\title[Applications of poset fiber theorems]{A poset fiber theorem for 
doubly Cohen-Macaulay posets and its applications to non-crossing 
partitions and injective words}
\author[M. Kallipoliti]{Myrto Kallipoliti}
\author[M. Kubitzke]{Martina Kubitzke}
\address{Fakult\"at f\"ur Mathematik, Universit\"at Wien, Garnisongasse 3, A-1090 Wien}
\email{myrto.kallipoliti@univie.ac.at, martina.kubitzke@univie.ac.at}

\begin{document}
\maketitle
\thispagestyle{empty}

\begin{abstract}
This paper studies topological properties of the lattices of non-crossing partitions of types A and B 
and of the
poset of injective words. Specifically, it is shown that after the removal of the bottom and top elements (if existent) 
these posets are doubly Cohen-Macaulay. This
strengthens the well-known facts that these posets are
Cohen-Macaulay. Our results rely on a new poset fiber theorem
which turns out to be a useful tool to prove double (homotopy)
Cohen-Macaulayness of a poset. Applications to complexes of
injective words are also included\end{abstract}

\section{Introduction and results}\label{Sec:Intro}
This paper focuses on the study of the topology of 
the lattices of non-crossing partitions of types $A$ and $B$ (denoted by $\NC^A(n)$ and $\NC^B(n)$, respectively) and 
the poset of injective words on $n$ letters (denoted by $\I_n$). 
In addition, we consider 
complexes of injective words, which were originally defined by Jonsson and Welker \cite{JW} and in special cases also by Ragnarsson and 
Tenner \cite{RT1,RT2}, and extend some of the known results for those cell complexes. 
Our results rely on a new technique for showing that a poset, \emph{i.e.}, its order complex, is doubly (homotopy) Cohen-Macaulay.  
Double Cohen-Macaulayness is known to be a topological property \cite[Theorem 9.8]{Walker}, which was originally introduced by Baclawski 
in \cite{Baclawski2}. 
A Cohen-Macaulay complex $\Delta$ is called \emph{doubly Cohen-Macaulay} if for every vertex $v\in\Delta$ the complex $\Delta-\{v\}$ 
is Cohen-Macaulay of the same dimension as $\Delta$. 
Particular interest in this class of complexes partly stems from the fact that those complexes are conjectured to satisfy the $g$-conjecture, see 
\emph{e.g.}, \cite[Problem 4.2]{Swartz} and \cite{Nevo} for partial results. 
There exists a variety of fairly well-studied complexes, \emph{e.g.}, 
homology spheres, reduced order complexes of geometric lattices \cite{Baclawski2}, finite buildings \cite{bjo3} and independence 
complexes of matroids \cite{Chari} that are known to be doubly Cohen-Macaulay.   
The latter three classes of complexes admit so-called \emph{convex ear decompositions} \cite{Chari}. Those decompositions 
were further established by Swartz 
\cite[Theorem 4.1]{Swartz} as maybe the main tool for proving double Cohen-Macaulayness of a complex. 
If one wants to show that a simplicial complex is Cohen-Macaulay, shellability might be considered the analogue of convex ear 
decompositions. Another method for proving (homotopy) Cohen-Macaulayness is provided by the classical poset fiber theorems of 
Baclawski \cite{Baclawski} and Quillen \cite{Q}. 
To the best of our knowledge, there do not exist analogues of these theorems for higher Cohen-Macaulay connectivity. 
We close this gap by providing the following novel poset fiber theorem for doubly homotopy Cohen-Macaulay intervals.
We recall that homotopy Cohen-Macaulayness is an homotopy version of the Cohen-Macaulay property and that every homotopy 
Cohen-Macaulay poset is Cohen-Macaulay. Similarly, doubly homotopy Cohen-Macualayness implies Cohen-Macaulayness.

\begin{theorem}
\label{th:PFT-2HCM-b}
Let $P$ be a graded poset, $I=(u,v)$ be an open interval in $P$ and $x\in I$. 
Assume that $I-\{x\}$ is graded and that $Q$ is a homotopy Cohen-Macaulay poset. 
Let further $f:P\to Q$ be a surjective rank-preserving poset map which satisfies the following conditions:
\begin{enumerate}
\item [(i)] For every $q\in Q$ the fiber $f^{-1}\left(\langle q\rangle\right)$ is homotopy Cohen-Macaulay. 
\item [(ii)] There exists $q_0\in Q$ such that 
\begin{itemize}
\item $f^{-1}(q_0)=\{x\}$ and $f(I)-\{q_0\}$
is homotopy Cohen-Macaulay, and
\item for every $q>q_0$ and $p\in f^{-1}(q)\cap I$ the poset $[u,p] -\{x\}$ is homotopy Cohen-Macaulay.
\end{itemize}
\end{enumerate}
Then $I-\{x\}$ is homotopy Cohen-Macaulay as well.
If for all $x\in I$ there exists a map satisfying the above conditions and if $\rank\left(I-\{x\}\right)=\rank(I)$, 
then $I$ is doubly homotopy Cohen-Macaulay.
\end{theorem}

In the above theorem $\langle q\rangle$ denotes the order ideal of $Q$ generated by the singleton $\{q\}$.
As a corollary of the above theorem, we derive a poset fiber theorem that extends 
Quillen's theorem for homotopy Cohen-Macaulay posets \cite[Corollary 9.7]{Q}. 

\begin{corollary}\label{th:PFT-2HCM}
Let $P$ be a graded poset without a minimum and a maximum element and let $x\in P$. Assume that $P-\{x\}$ is graded and that $Q$ 
is a homotopy Cohen-Macaulay poset. 
Let further $f:P\to Q$ be a surjective rank-preserving  poset map which satisfies the following conditions:
\begin{enumerate}
\item [(i)] For every $q\in Q$ the fiber $f^{-1}\left(\langle q\rangle\right)$ is homotopy Cohen-Macaulay. 
\item [(ii)] There exists  $q_0\in Q$ such that 
\begin{itemize}
\item $f^{-1}(q_0)=\{x\}$ and $Q-\{q_0\}$ 
is homotopy Cohen-Macaulay, and
\item for every $q>q_0$ and $p\in f^{-1}(q)$ the poset $\langle p\rangle -\{x\}$ is homotopy Cohen-Macaulay.
\end{itemize}
\end{enumerate}
Then $P-\{x\}$ is homotopy Cohen-Macaulay as well. 
If for all $x\in P$ there exists a map satisfying the above conditions and if $\rank(P-\{x\})=\rank(P)$, 
then $P$ is doubly homotopy Cohen-Macaulay.
\end{corollary}

We will further give a generalization of Corollary \ref{th:PFT-2HCM} to posets having higher Cohen-Macaulay connectivity, see Proposition \ref{th:PFTk-HCM}. 

The original motivation of Theorem \ref{th:PFT-2HCM-b} comes from the objective to investigate double Cohen-Macaulayness of  
the lattices of non-crossing partitions in type $A$ and $B$ and 
the poset of injective words. Since we were not able to successfully attach this problem using rather evolved techniques as convex ear decompositions 
or classical poset fiber theorems, we needed to develop a new methodology. 

In the past, the lattice of non-crossing partitions of a finite Coxeter group as well as the poset 
of injective words have attracted the attention of a lot of different researchers and are fairly well-studied objects. 

The poset of non-crossing partitions $\NC(W)$ for a finite Coxeter group $W$ has been studied extensively 
and it has been shown to be a graded, self-dual lattice \cite{bessis}. 
In 1980, Bj\"orner and Edelman \cite[Example 2.9]{bjo1} constructed an 
EL-shelling of $\NC^A(n)$ and in 2002, Reiner \cite{R} proved 
the same result for non-crossing partitions of type B. Finally, 
EL-shellability of $\NC(W)$ was verified for all types of finite Coxeter groups 
by Athanasiadis, Brady and Watt \cite[Theorem 1.1]{ABW} who were 
able to provide a case-independent proof. In particular, it follows from this result 
that $\NC(W)$ is homotopy Cohen-Macaulay. 
In personal communication, Athanasiadis proposed to study the 
problem of whether $\NC^A(n)$ and $\NC^B(n)$ are doubly (homotopy) 
Cohen-Macaulay. Using Theorem \ref{th:PFT-2HCM-b} we can give an 
affirmative answer to this question. 
In fact, we provide a uniform proof for both types.

\begin{theorem}\label{th:NC}
The proper parts of the lattices of non-crossing partitions $\NC^A(n)$ and $\NC^B(n)$ are doubly homotopy Cohen-Macaulay for all $n\geq 3$.
\end{theorem}

Maybe of a little bit less interest than the lattices of non-crossing partitions but still of fairly much interest 
is the poset of injective words 
Already in 1978, Farmer \cite{Farmer} showed that the regular CW-complex $\Gamma_{n}$ whose face poset is $\I_{n+1}$ is 
homotopy equivalent to a wedge of spheres of top dimension. 
Some years later, Bj\"orner and Wachs \cite[Theorem 6.1.(i)]{BW2} 
could strengthen this result by demonstrating that the complex 
$\Gamma_n$ is even CL-shellable. More recently, Reiner and Webb \cite{RW} 
computed the homology of $\Gamma_n$ as an $S_{n+1}$-module, 
and Hanlon and Hersh \cite{HW} provided a refinement of this result 
by giving a Hodge type decomposition for the homology of $\Gamma_n$. 
During a discussion, Athanasiadis suggested to investigate the topology of the poset $\I_n-\{\emptyset,x\}$, 
where $\emptyset$ denotes the 
empty word of $\I_n$ and $x\in \I_n$ can be any word different from $\emptyset$.
In this work, using Corollary \ref{th:PFT-2HCM}, we show that the posets 
$\I_n-\{\emptyset,x\}$, 
\textit{i.e.}, their order complexes, are homotopy Cohen-Macaulay. In particular, this yields the following result.

\begin{theorem}\label{th:injective2Constructible}
Let $n\geq 2$ and let $\emptyset\in \I_n$ denote the empty word. Then $\I_n-\{\emptyset\}$ 
is doubly homotopy Cohen-Macaulay.
\end{theorem}

In \cite{JW}, several generalizations and restrictions of the 
CW-complex $\Gamma_n$ are introduced and further investigated. 
Jonsson and Welker associate to a given simplicial complex $\Delta$ 
several so-called \emph{complexes of injective words}, which are subcomplexes 
of $\Gamma_n$ and which depend on a certain poset $P$ and a graph $G$, respectively 
(see Section \ref{subsect:Injective} for the precise definitions). 
It is shown in \cite{JW} that these complexes are Boolean cell complexes. 
Furthermore, using the poset fiber theorems for 
sequentially (homotopy) Cohen-Macaulay posets \cite[Theorem 5.1]{BWW}, 
it is proved that sequentially (homotopy) Cohen-Macaulayness is preserved 
under those constructions, see \cite[Theorem 1.3]{JW}. 
In \cite{RT1,RT2}, Ragnarsson and Tenner considered, what they call, 
\emph{Boolean complexes of Coxeter systems}. Those are complexes of 
injective words in the sense of Jonsson and Welker, 
where the underlying simplicial complex and graph are the full simplex and 
the Coxeter graph of a Coxeter system, respectively. 
Ragnarsson and Tenner show that that these complexes are homotopy equivalent to a 
wedge of top-dimensional spheres and compute the number of spheres appearing in the wedge. 
The first part of this result also follows from \cite{JW}. 

In a conversation with Welker, he raised the question of whether one can use Theorem \ref{th:injective2Constructible} to
show analogues of Jonsson's and his results \cite[Theorem 1.3]{JW}, assuming that the underlying simplicial complex is 
doubly homotopy Cohen-Macaulay. We give the following answer to his question.

\begin{theorem}
\label{th:general}
 Let $\Delta$ be a doubly homotopy Cohen-Macaulay simplicial complex on the vertex set $[n]=\{1,\ldots,n\}$.
\begin{itemize}
 \item[(i)] If $P=([n],\preceq_P)$ is a poset, then the Boolean cell complex $\Gamma(\Delta,P)$ is doubly homotopy Cohen-Macaulay. 
\item[(ii)] If $G=([n],E)$ is a graph on vertex set $[n]$, then the Boolean cell complex $\Gamma/G(\Delta)$ is doubly homotopy Cohen-Macaulay.
\end{itemize}
\end{theorem}

It is worth noting and in a certain extent astonishing that the proof of this theorem does not use Theorem \ref{th:injective2Constructible}, 
but is a direct application of Corollary \ref{th:PFT-2HCM} 
to the same maps which were used by Jonsson and Welker in \cite{JW} to prove their Theorem 1.3.

The paper is structured as follows. 
Section \ref{subsect:PosetComplex} reviews background on posets and simplicial complexes and most of the
terminology and concepts which have been used in the introduction are explained within this section. 
In Section \ref{subsect:NC}, we recall the definitions and some properties of non-crossing partition lattices, with a special emphasis on 
non-crossing partitions of types A and B. 
Sections \ref{subsect:Injective} and \ref{subsect:CompInj} fulfill the same task for the poset and complexes of injective words, respectively. 
Section \ref{sec:ConstructibilityPosetFiber} focuses on poset fiber theorems.
In the first part, we give the proofs of the poset fiber theorems for doubly homotopy Cohen-Macaulay intervals and posets 
(Theorem \ref{th:PFT-2HCM-b} and Corollary \ref{th:PFT-2HCM}, respectively). 
In the second half of this section, we state and prove a poset fiber theorem (Theorem \ref{th:StrgConstrFiberTheorem}) for 
strongly constructible posets, a notion which was introduced in \cite{ca}. 
Subsequently, we apply this theorem to the poset of injective words, thereby providing a direct proof 
that this poset in strongly constructible.
In Section \ref{sect:Application2}, Theorem \ref{th:PFT-2HCM-b} is employed to prove 
double homotopy Cohen-Macaulayness of the non-crossing partition lattices $\NC^A(n)$ and $\NC^B(n)$ (Theorem \ref{th:NC}). 
In Section \ref{sect:Application1}, we use Corollary \ref{th:PFT-2HCM} to show that $\I_n$ 
is doubly homotopy Cohen-Macaulay (Theorem \ref{th:injective2Constructible}). 
Another application of Corollary \ref{th:PFT-2HCM} is provided by Theorem \ref{th:general}, which is the 
natural extension of Theorem 1.3 in \cite{JW} to doubly homotopy Cohen-Macaulay complexes.


\section{Preliminaries}
\subsection{Partial orders and simplicial complexes} \label{subsect:PosetComplex}
Let $(P,\leq)$ be a finite partially ordered set (poset for short) and let $x,y\in P$. We say that $y$
\emph{covers} $x$ and write $x\to y$, if $x<y$ and if there is
no $z\in P$ such that $x< z< y$. The poset $P$ is called \emph{bounded}, if there exist elements
$\hat{0}$ and $\hat{1}$ such that $\hat{0}\leq x\leq \hat{1}$ for every $x\in P$. 
The \emph{proper part} $\bar{P}$ of a bounded poset $P$ is the subposet obtained after removing $\hat{0}$ and $\hat{1}$, 
\emph{i.e.}, $\bar{P}=P-\{\hat{0},\hat{1}\}$.
A subset $C$ of a poset $P$
is called a \emph{chain}, if any two elements of $C$ are comparable in $P$. 
Throughout this paper, we denote by $\{\hat{0},\hat{1}\}$ the $2$-element chain, with
$\hat{0}<\hat{1}$. The \emph{length} of a (finite)
chain $C$ is equal to
$|C|-1$. We say that $P$ is \emph{graded}, 
if all maximal chains of $P$ have the same length and call this common length the \emph{rank} of $P$, denoted by $\rank(P)$. Moreover, 
assuming that $P$ has a minimum $\hat{0}$, 
there exists a unique function $\rank:P\to \mathbb{N}$, called the \emph{rank function} of $P$, such that
\[\rank(y)=\left\{
\begin{array}{ll}
0, & \mbox{if $y=\hat{0}$}, \\
\rank(x)+1, & \mbox{if $x\to y$}.
\end{array}
\right.\] 
We say that $x$ has \emph{rank} $i$, if $\rank(x)=i$. 
For $x\leq y$ in $P$ we denote by $[x,y]_P$ the closed interval 
$\{z\in P~:~x\leq z\leq y\}$ of $P$, endowed with the partial order induced by $P$. 
For $S\subseteq P$, the \emph{order ideal} of $P$ generated by
$S$ is the subposet $\langle S\rangle_P=\{x\in P~:~x\leq y \mbox{ for some }y\in S\}$. 
We write $\langle y_1,y_2,\dots,y_m \rangle$ for the order ideal of $P$ generated by the set
$\{y_1,y_2,\dots,y_m\}$.
For intervals, as well as for order ideals, we use the convention that the subscript $P$ is omitted, when it is clear from 
the context in which poset $P$ a certain interval or ideal is considered. 
For $x\in P$ we set $P_{<x}=\{p\in P~:~p<x\}$. 
Given two posets $(P,\leq_P)$ and $(Q,\leq_Q)$, a map $f:P\rightarrow Q$ is called a
\emph{poset map} if it is order-preserving, \textit{i.e}., $x\leq_Py$ implies $f(x)\leq_Qf(y)$ for all
$x,y\in P$. If, in addition, $f$ is a bijection with order-preserving inverse, then $f$ is said to be a 
\emph{poset isomorphism}. In this case, the posets $P$
and $Q$ are said to be \emph{isomorphic}, and we write $P\cong Q$. 
Assuming that $P$ and $Q$ are graded, 
a map $f:P\rightarrow Q$ is called \emph{rank-preserving}, if for every $x\in P$, 
the rank of $f(x)$ in $Q$ is equal to the rank of $x$ in $P$, \emph{i.e.}, $\rank(f(x))=\rank(x)$. 
The \emph{dual} of a poset $(P,\leq_P)$ is the poset $(P^{*},\leq_{P^*})$ on the same ground set as $P$ with reversed order relations,
\textit{i.e.}, $x\leq_{P^{*}} y$ if and only if $y\leq_P x$.
A poset $P$ is called \emph{self-dual} if $P\cong P^*$, and it is \emph{locally self-dual} if every closed interval
of $P$ is self-dual. 
The \emph{direct product} of two posets 
$P$ and $Q$ is the poset $P\times Q$ on the set $\{(x,y)~:~x\in P,\, \ y\in Q\}$, 
for which $(x,y)\leq (x',y')$ holds in $P\times Q$, if $x\leq_P x'$ and $y\leq_Q y'$. 
The \emph{ordinal sum} $P\oplus Q$ of $P$ and $Q$ is the poset defined on the disjoint union of $P$ and $Q$ with the order relation 
$x\leq y$, if (i) $x,y\in P$ and $x\leq_P y$, or (ii) $x,y\in Q$ and $x\leq_Q y$, or (iii) $x\in P$ and $y\in Q$.
For more information on partially ordered sets, we refer the reader to \cite[Chapter 3]{St}.

An \emph{abstract simplicial complex} $\Delta$ on a finite vertex set $V$ is a collection
of subsets of $V$ such that $G \in \Delta$ and $F\subseteq G$ imply $F\in\Delta$. 
The elements of $\Delta$ are called \emph{faces}. Inclusionwise maximal and $1$-element faces are called \emph{facets}
and \emph{vertices}, respectively.
The \emph{dimension} of a face $F\in \Delta$ is equal to $|F|-1$ and is denoted by $\dim(F)$.
The \emph{dimension} of $\Delta$ is defined to be the maximum dimension of a face of $\Delta$ and is denoted by $\dim\Delta$.
If all facets of $\Delta$ have the same dimension, then $\Delta$ is called \emph{pure}.
The \emph{link} of a face $F$ of 
$\Delta$ is defined as $\lk_{\Delta}(F)=\{G~:~F\cup G\in\Delta,\;F\cap G=\emptyset\}$. 
A simplicial complex $\Delta$ is 
\emph{homotopy Cohen-Macaulay}, if for all $F\in\Delta$ the link of $F$ is topologically $(\dim
(\lk_{\Delta} (F))-1)$-connected. A pure $d$-dimensional simplicial complex $\Delta$ is \emph{shellable}, if there exists 
a linear order $F_1,\ldots, F_m$ of the facets of $\Delta$ such that $\langle F_i\rangle\cap\langle F_1,\ldots,F_{i-1}\rangle$ is 
generated by a non-empty set of maximal proper faces of $\langle F_i\rangle$ for all $2\leq i\leq m$. Here, $\langle F_i\rangle$ and 
$\langle F_1,\ldots,F_{i-1}\rangle$ denote the simplicial complexes whose faces are subsets of $F_i$ and $F_1,\ldots,F_{i-1}$, respectively.
We recall that the Cohen-Macaulay property is defined in an analogue way, if one replaces the homotopy groups 
with homology groups. Cohen-Macaulayness is a topological property and it is implied by homotopy Cohen-Macaulayness. 
For a $d$-dimensional simplicial complex we have the following hierarchy of properties: 
shellable $\Rightarrow$  homotopy Cohen-Macaulay $\Rightarrow$
homotopy equivalent to a wedge of $d$-dimensional spheres. 
Additional background concerning the topology of simplicial complexes can be found in \cite{bjo2} and \cite{mwa}.

To every poset $P$ one can associate its so-called \emph{order complex} $\Delta(P)$, 
which is an abstract simplicial complex on vertex set $P$ whose $i$-dimensional faces are the chains
of $P$ of length $i$. If $P$ is graded of rank $n$, then the order complex $\Delta(P)$ is pure of dimension $n$. 
If we speak about a topological property of $P$, we mean the corresponding property of $\Delta(P)$. 
Likewise, we say that $P$ is homotopy Cohen-Macaulay and shellable, if $\Delta(P)$ is 
homotopy Cohen-Macaulay and shellable, respectively.

\subsection{Non-crossing partitions} \label{subsect:NC}
Let $W$ be a finite Coxeter group and let $T$ denote the set of all reflections in $W$.
Given $w \in W$, the \emph{absolute length} $\ell_T (w)$ of $w$ is the smallest integer $k$ such that $w$ can be
written as a product of $k$ elements of $T$. 
The \emph{absolute order} $\Abs(W)$ is the partial order $\mik$ on $W$ defined by,
\[ u \mik v \ \ \ \mbox{if and only if} \ \ \ \ell_T(u) \, + \,
\ell_T(u^{-1} v) \, = \, \ell_T (v) \]
for $u, v \in W$. Equivalently, $\preceq$ is the partial order on $W$ with covering relations $w
\to wt$, where $w \in W$ and $t \in T$ are such that $\ell_T (w) <
\ell_T (wt)$. The poset $\Abs(W)$ is 
graded with a minimum element $e$ and rank function $\ell_T$, see \emph{e.g.}, \cite{arm, bessis}.
If $c$ is a Coxeter element of $W$, then the interval
\begin{equation*}
\NC(W,c)=[e,c]=\{w\in W~:~e\leq_T w\leq_T c\}
\end{equation*}
is called the lattice of \emph{non-crossing partitions}. 
It is well-known (see \emph{e.g.}, \cite[Section 2.6]{arm}) that for Coxeter elements $c,c'\in W$ it holds
that $\NC(W,c)\cong \NC(W,c')$. We therefore often suppress $c$ from the notation and write $\NC(W)$ instead.
It follows from \cite[Lemma 2.5.4]{arm} that $\Abs(W)$ is locally self-dual for every finite Coxeter group $W$. 
As a consequence, the following corollary holds. 

\begin{corollary}\label{cor:selfDual}
Let $W$ be a finite Coxeter group with set of reflections $T$. Then, for all $u\in P$ the principal 
lower order ideal $\langle u\rangle$ is self-dual. In particular, $\NC(W)$ is self-dual. 
\end{corollary}

In the following two paragraphs, we give a more detailed description of the lattices of non-crossing partitions for the 
symmetric group $S_n$ and the hyperoctahedral group $B_n$.

\subsubsection{Non-crossing partitions of type A} \label{subsect:NCA}
Let $W$ be the symmetric group $S_n$. We view this group as the group of permutations of the set $\{1,2,\dots,n\}$. 
The set of reflections $T$ consists of 
all transpositions $(ij)$ for $1\leq i<j\leq n$, and the Coxeter elements of $S_n$ are the $n$-cycles of $S_n$. 
The absolute length of an element of $S_n$ equals $n$ minus the number
of cycles in its cycle decomposition. This in particular means that $\Abs(S_n)$ has rank $n-1$. 
In \cite[Section 2]{Brady} the following description of the absolute order was provided: 
For all $u,v\in S_n$, we have $u\leq_T v$ if and only if 
\begin{itemize}
\item[(i)] every cycle in the cycle decomposition of $u$ can be obtained from some cycle in the cycle decomposition
of $v$ by deleting elements, and
\item[(ii)] any two cycles $a$ and $b$ of $u$, which are obtained from the same cycle $c$ of $v$, are non-crossing with respect to $c$.
\end{itemize}
Here, disjoint cycles $a$ and $b$ are called \emph{non-crossing} with respect to $c$, if there does not exist a cycle $(ijkl)$ which
is obtained from $c$ by deleting elements such that $i,k$  and $j,l$ are elements of $a$ and $b$, respectively.

Consider the Coxeter element $c=(12\cdots n)$. We denote by $\NC^A(n)$ the poset of non-crossing partitions of $S_n$ associated to $c$, 
and we call its elements \emph{non-crossing partitions of type A}. 
Figure \ref{a34} illustrates the Hasse diagrams of the posets $\NC^A(3)$ and $\NC^A(4)$.

\begin{figure}[h]
\begin{center}
\includegraphics[width=4in]{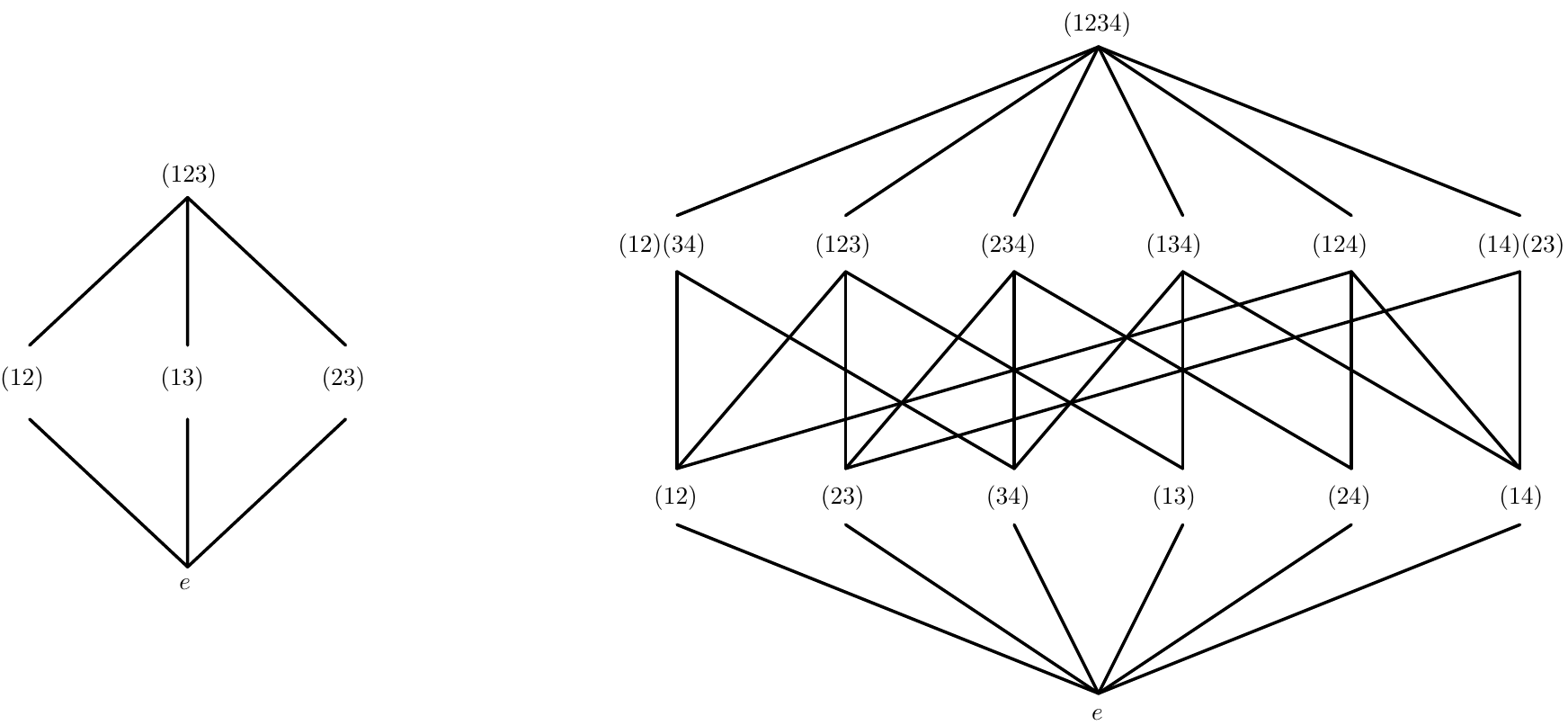}
\end{center}
\caption{The posets $\NC^A(3)$ and $\NC^A(4)$.}
\label{a34}
\end{figure}

\subsubsection{Non-crossing partitions of type B} \label{subsect:NCB}
Let $W$ be the hyperoctahedral group $B_n$. This group can be thought of as the group of signed permutations of the set 
$\{1,2,\dots,n\}$. 
These are permutations $\tau$ of $\{\pm1\pm2,\dots,\pm n\}$, subject to the condition, that $\tau(-i)=-\tau(i)$ for all $1\leq i\leq n$. 
For signed permutations, two types of cycles are usually distinguished. Cycles of the form $(a_1 a_2\cdots a_k)(-a_1-a_2\cdots -a_k)$ 
are called \emph{paired $k$-cycles} and denoted by $\lleft a_1, a_2,\dots, a_k\rright$. Cycles of the form $(a_1a_2\cdots a_k -a_1 -a_2\cdots -a_k)$ 
are referred to as \emph{balanced $k$-cycles} and abbreviated by $[a_1, a_2,\dots, a_k]$. 
The set of all reflections of $B_n$ consists of all balanced $1$-cycles  
$[i]$ for $1\leq i\leq n$ and the paired $2$-cycles $\lleft i,\pm j\rright$ for $1\leq i<j\leq n$. 
The Coxeter elements of $B_n$ are the balanced $n$-cycles of $B_n$.
The absolute length of an element of $B_n$ equals $n$ minus the number
of paired cycles in its cycle decomposition. This in particular means that $\Abs(B_n)$ has rank $n$. 
Covering relations $w\to wt$ in $\Abs(B_n)$, where 
$w$ and $t$ are non-disjoint cycles, can be described by an explicit set of conditions (see \emph{e.g.}, \cite[Section 2.2]{Myrto}).

Consider the Coxeter element $c=[1,2,\dots, n]$. We denote by $\NC^B(n)$ the poset of non-crossing partitions 
of $B_n$, associated to $c$, 
and we call its elements \emph{non-crossing partitions of type B}.  
Figure \ref{s34} illustrates the Hasse diagram of the poset $\NC^B(2)$.

\begin{figure}[h]
\begin{center}
\includegraphics[width=1.3in]{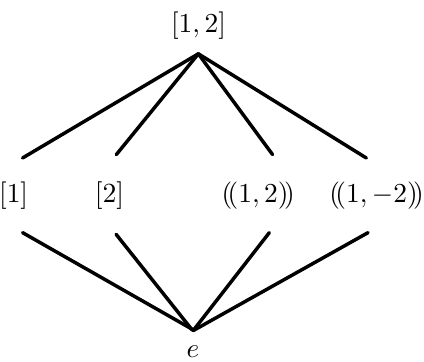}
\end{center}
\caption{The poset $\NC^B(2)$.}
\label{s34}
\end{figure}

For further information on Coxeter groups and non-crossing partitions, we refer to \cite{arm}.

\subsection{The poset of injective words} \label{subsect:Injective}
A word $\omega$ over a finite alphabet $A$ is called \emph{injective}, if no letter
appears more than once. We denote by $\I_n$ the set
of injective words on $[n]=\{1,\ldots,n\}$. The order relation on $\I_n$ is given by the containment of subwords, \textit{i.e.},
$\omega_1\cdots \omega_s \leq\sigma_1\cdots \sigma_r$, if and only if there exist $1\leq i_1<i_2<\cdots <i_s\leq r$ such 
that $\omega_j=\sigma_{i_j}$ for $1\leq j\leq s$. 
\textit{E.g.}, we have $124< 12345$ in $\I_5$, whereas $12$ and $23$ are incomparable in $\I_n$ for $n\geq 3$. 
It is a rather classical result that $\I_n$ is shellable and in particular homotopy Cohen-Macaulay \cite[Theorem 6.1.(i)]{BW2}. 
Moreover, every closed interval of $\I_n$ is isomorphic to a Boolean algebra \cite{Farmer} and in particular shellable. 
Figure \ref{i23} illustrates the Hasse diagrams of the posets $\I_2$ and $\I_3$. 
\begin{figure}[h]
\begin{center}
\includegraphics[width=3.5in]{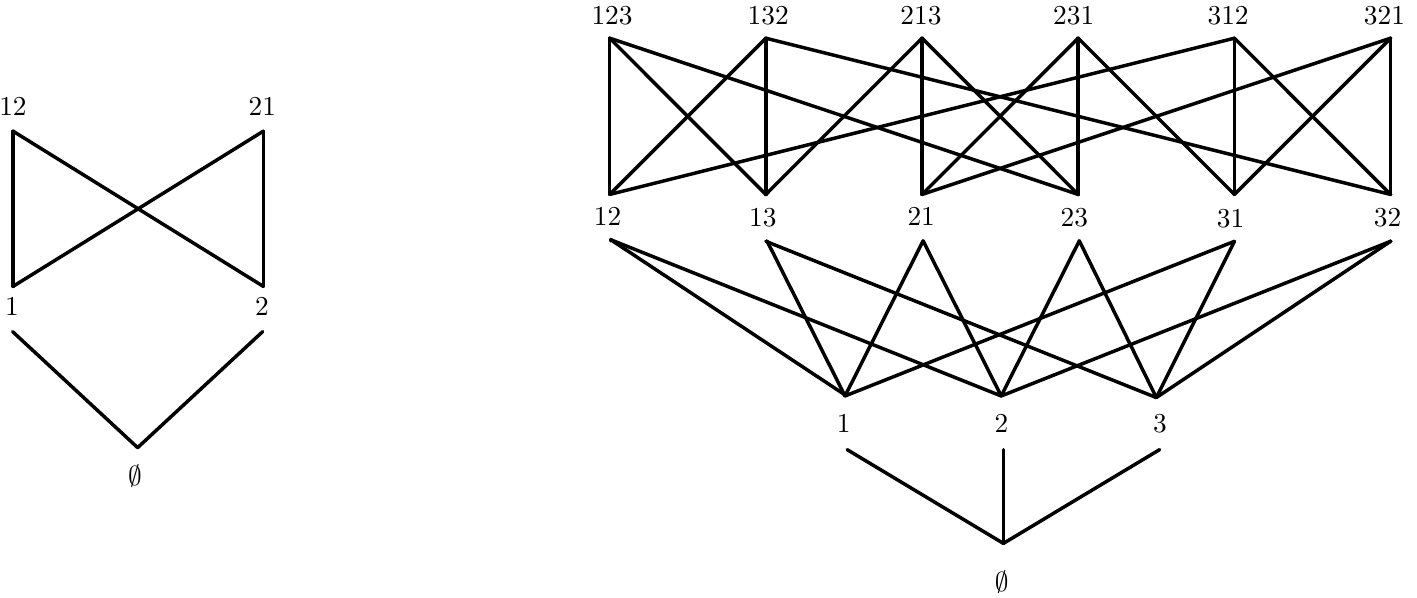}
\end{center}
\caption{The posets $\I_2$ and $\I_3$.}
\label{i23}
\end{figure}

\subsection{Complexes of injective words} \label{subsect:CompInj}
It is a well-known fact that $\I_{n}$ is the face poset of a Boolean cell complex \cite{Farmer}. So as to distinguish between the poset of injective words and the 
corresponding cell complex, we adapt the notations from \cite{JW} and use $\Gamma_{n}$ to denote the complex determined by $\I_{n+1}$. 
Each $d$-cell of $\Gamma_n$ corresponds to an injective word $w$ of length $d+1$ and the faces of such a cell are given by the subwords of $w$. 
As mentioned in Section \ref{Sec:Intro}, Jonsson and Welker \cite{JW} and in a more restricted setting also 
Ragnarsson and Tenner \cite{RT1,RT2}, considered several generalizations of the complex $\Gamma_n$. We now provide the constructions 
of those complexes. 
To simplify notation, for a word $w=w_1\cdots w_s\in \I_n$, we set $c(w)=\{w_1,\ldots,w_s\}$ and call this the \emph{content} of $w$. 

\begin{definition}
 Let $\Delta$ be a simplicial complex on vertex set $[n+1]$.
\begin{itemize}
 \item[(i)] The complex $\Gamma(\Delta)$ is the restriction of $\Gamma_n$ to words whose content is a face of $\Delta$, \emph{i.e.},
\begin{equation*}
 \Gamma(\Delta)=\{w\in \Gamma_n~ :~ c(w)\in \Delta\}.
\end{equation*}
\item[(ii)] Let $P=([n+1],\leq_P)$ be a poset on ground set $[n+1]$. The complex $\Gamma(\Delta,P)$ is the subcomplex of 
$\Gamma(\Delta)$ satisfying the 
following condition:
\begin{equation*}
 w=w_1\cdots w_s\in \Gamma(\Delta,P) \mbox{ and } w_i<_P w_j \;\Rightarrow\; i<j.
\end{equation*}
\item[(iii)] Let $G=([n+1],E)$ be a graph on vertex set $[n+1]$ with edge set $E$. 
The equivalence class $[w]$ of an injective word $w\in \Gamma_n$ 
contains all words $v$ that can be obtained from $w$ by applying a sequence 
of commutations $ss'\rightarrow s's$ such that $\{s,s'\}\notin E$. 
The set of equivalence classes $[w]$ of injective words $w\in \Gamma(\Delta)$ is denoted by $\Gamma/G(\Delta)$. 
An ordering on $\Gamma/G(\Delta)$ is defined 
by setting $[v]\preceq [w]$, if there exist representatives $v'\in[v]$ and $w'\in[w]$ such that $v'\leq w'$ in $\I_{n+1}$.
\end{itemize}
\end{definition}

It directly follows from the definitions that $\Gamma(\Delta,P)$ 
is a subcomplex of $\Gamma(\Delta)$ and these two complexes coincide, if $P$ is an antichain.
If, in contrast, $P$ is a total order, then it holds that $\Gamma(\Delta,P)\cong \Delta$. 
It is shown in \cite{JW} that all three complexes, $\Gamma(\Delta)$, $\Gamma(\Delta,P)$ and $\Gamma/G(\Delta)$, are Boolean cell complexes. 
Furthermore, if $\Delta$ is shellable and $G$ is a simple graph, then (sequentially) homotopy 
Cohen-Macaulayness and shellability are maintained after performing any of those 
constructions \cite[Theorem 1.3, Theorem 1.2]{JW}. 
In the special case of a full simplex $\Delta$ and the Coxeter graph $G$ 
of a Coxeter system, shellability also follows from Remark 5.11 in \cite{RT1}.


\section{Poset fiber theorems} \label{sec:ConstructibilityPosetFiber}

In this section, we focus on the proofs of the poset fiber theorems for doubly homotopy Cohen-Macaulay intervals and posets. 
Furthermore, we state and prove a poset fiber theorem for strongly constructible posets and give an application to injective words. 
These theorems are inspired by the following classical poset fiber theorem of Quillen. 

\begin{theorem} \emph{\cite[Corollary 9.7]{Q}}
\label{Quillen}
 Let $P$ and $Q$ be graded posets. Let further $f:P\to Q$ be a surjective rank-preserving poset map.
Assume that for every $q\in Q$ the fiber $f^{-1}\left(\langle q\rangle\right)$ is homotopy Cohen-Macaulay. 
If $Q$ is homotopy Cohen-Macaulay, then so is $P$.
\end{theorem}

\subsection{Poset fiber theorems for doubly homotopy Cohen-Macaulay posets} \label{subsec:2HCM}

Before providing the proof of Theorem \ref{th:PFT-2HCM-b}, 
we recall the notion of doubly homotopy Cohen-Macaulay posets.

\begin{definition}
A poset $P$ is called \emph{doubly homotopy Cohen-Macaulay} if $P$ is homotopy Cohen-Macaulay and 
if for every $x\in P$ the poset $P-\{x\}$ is homotopy Cohen-Macaulay of the same rank as $P$. 
\end{definition}

We will use the following result which is implied by Remark 2.6 and Corollary 3.2 in \cite{BWW}.

\begin{corollary}\label{cor:connected}
Let $P$ and $Q$ be graded posets of rank $n$.
Let $f:P\to Q$ be a surjective rank-preserving poset map such that for all $q\in Q$ 
the order complex $\Delta(Q_{>q})$ is $(n-\rank(q)-2)$-connected and
for all non-minimal $q\in Q$ the inclusion map 
\begin{equation*}
\Delta\left(f^{-1}(Q_{<q})\right)\hookrightarrow \Delta\left(f^{-1}\left(\langle q\rangle\right)\right)
\end{equation*}
is homotopic to a constant map which sends $\Delta\left(f^{-1}(Q_{<q})\right)$ to $c_q$ for some 
$c_q\in \Delta\left(f^{-1}\left(\langle q\rangle\right)\right)$.
Then $\Delta(P)$ is $(n-1)$-connected, if and only if $Q$ is $(n-1)$-connected.
\end{corollary}

\noindent{\bf{Proof of Theorem \ref{th:PFT-2HCM-b}}.} 
It directly follows from Theorem \ref{Quillen} that the poset $P$ is homotopy Cohen-Macaulay and 
hence so is the interval $I=(u,v)$. 
Let $\widetilde{I}$ denote the poset $I-\{x\}$ and let $k$ be its rank.  
We need to verify that all links of faces $F\in\Delta(\widetilde{I})$ are 
$(\dim(\lk_{\Delta(\widetilde{I})}(F))-1)$-connected. 
The arguments we use are similar to those employed in the proof of \cite[Theorem 5.1 (i)]{BWW}. 
  
We first show that $\Delta(\widetilde{I})=\lk_{\Delta(\widetilde{I})}(\emptyset)$ is $(k-1)$-connected. For this aim, we use 
Corollary \ref{cor:connected}.

Let $\widetilde{f}:\widetilde{I} \rightarrow f(I)-\{q_0\}$ denote the restriction of $f$ to $\widetilde{I}$. 
This map is well-defined, since $f^{-1}(q_0)=\{x\}$, and it is a surjective poset 
map, because $f$ is. Since $f$ is rank-preserving and since $\widetilde{I}$ is graded by hypothesis, we deduce 
that $\widetilde{f}$ is rank-preserving. 
We set $\widetilde{J}=f(I)-\{q_0\}$ and by assumption we know that $\widetilde{J}$ is homotopy Cohen-Macaulay.
In the following, consider $q\in \widetilde{J}$.
Since $\Delta(\widetilde{J}_{>q})$ is the link of a face of $\Delta(\widetilde{J})$, 
we infer from the above that 
$\Delta(\widetilde{J}_{>q})$ is $(\rank(\widetilde{J}_{>q})-1)=(\rank(f(v))-\rank(q)-3)$-connected. This shows one of the conditions  
of Corollary \ref{cor:connected} we need to verify. 
By assumption on $f$, the fiber $f^{-1}\left(\langle q\rangle\right)$ is homotopy Cohen-Macaulay 
and therefore it is $(\rank(q)-1)$-connected. 
As in the proof of Theorem 1.1 in \cite{BWW}, it follows that there exists a homotopy 
from the inclusion map 
$\Delta(f^{-1}(Q_{<q}))\hookrightarrow \Delta(f^{-1}\left(\langle q\rangle\right))$
to the constant map which sends $\Delta(f^{-1}(Q_{<q}))$ to $c_q\in \Delta(f^{-1}\left(\langle q\rangle\right))$. 
We can choose $c_q\in \Delta(\widetilde{f}^{-1}\left(\langle q\rangle\right))\subseteq \widetilde{I}$. 
Then the above homotopy restricts to a homotopy from 
$\Delta(\widetilde{f}^{-1}(\widetilde{J}_{<q}))\hookrightarrow \Delta(\widetilde{f}^{-1}\left(\langle q\rangle\right))$
to the constant map which sends $\Delta(\widetilde{f}^{-1}(\widetilde{J}_{<q}))$ to $c_q$. Thus, 
$\Delta(\widetilde{f}^{-1}(\widetilde{J}_{<q}))\hookrightarrow \Delta(\widetilde{f}^{-1}\left(\langle q\rangle\right))$ is
homotopic to a constant map. Finally, we can apply the Corollary aforementioned. Since, by homotopy Cohen-Macaulayness, $\widetilde{J}$ 
is $(k-1)$-connected, it follows that $\widetilde{I}$ is $(k-1)$-connected as well. 

It remains to show that all links of proper faces $F\neq \emptyset$ of $\Delta(\widetilde{I})$ 
are $(\dim(\lk_{\Delta(\widetilde{I})}(F))-1)$-connected.
Since the join of an $s$-connected and an $r$-connected complex is $(r+s-2)$-connected, 
it suffices to check open intervals and principal upper and 
lower order ideals (see \emph{e.g.}, \cite{BWW2}).

Let $(a,b)$ be an open interval in $\widetilde{I}$. Note that $(a,b)_P=(a,b)_{I}$. 
If $x\notin (a,b)_P$, then $(a,b)_{I}$ and $(a,b)_{\widetilde{I}}$ coincide.
Since $I$ is homotopy Cohen-Macaulay, it follows that $(a,b)_{\widetilde{I}}$ is $(\rank(b)-\rank(a)-3)$-connected.
Now let $a<x<b$ and let $c=f(b)$, \emph{i.e.}, $b\in f^{-1}(c)$. 
From $b\neq v$, we infer that $b\in I$ and thus $b\in f^{-1}(c)\cap I$. 
Moreover, we have $c>q_0$ and by condition (ii) of the theorem, it follows that $[u,b]_P-\{x\}$ 
is homotopy Cohen-Macaulay. 
Since $(a,b)_{\widetilde{I}}=(a,b)_P-\{x\}$ is the link of a face of $[u,b]_P-\{x\}$, we conclude  
that $(a,b)_{\widetilde{I}}$ is $(\rank(b)-\rank(a)-3)$-connected.
The same reasoning shows that open principal lower order ideals $\widetilde{I}_{<p}$ of $\widetilde{I}$ are $(\rank(p)-\rank(u)-3)$-connected.

Next, we show that for all $p\in \widetilde{I}$ the open principal upper order ideal $\widetilde{I}_{>p}=(p,v)_P-\{x\}$ is 
$(\rank(v)-\rank(p)-3)$-connected. 
If $p\nless x$, then $(p,v)_P-\{x\}=(p,v)_P$, and the claim follows, because $P$ is homotopy Cohen-Macaulay.
Let now $p < x$. 
We consider the restriction of $f$ to $P_{\geq p}$. To avoid confusion, 
let $\bar{f}: P_{\geq p}\to Q_{\geq f(p)}$ denote this restriction. 
We show that the map $\bar{f}$ is a surjective rank-preserving poset map, satisfying all assumptions of the theorem for the interval $(p,v)$ 
and the element $x\in(p,v)$. 
Since, due to $u<p$, we have $\rank([p,v]_P-\{x\})<\rank([u,v]_P-\{x\})$, we can then deduce by induction on the rank of the considered interval 
that $(p,v)_P-\{x\}=\widetilde{I}_{>p}$ is homotopy Cohen-Macaulay. 
In particular, we obtain that $\widetilde{I}_{>p}$ is $(\rank(v)-\rank(p)-3)$-connected. 
For the verification of the assumptions, first note that $Q_{\geq f(p)}$ is homotopy Cohen-Macaulay because $Q$ is. 
Clearly, $x\in (p,v)_P\subsetneq P_{\geq p}$. Since $\tilde{I}$ is graded by assumption, the same is true for $(p,v)_P-\{x\}$. 
Furthermore, $f$ is a rank-preserving poset map, thus so is $\bar{f}$. 
To see that $\bar{f}$ is surjective, let $q\in Q_{\geq f(p)}$. 
Since $f$ is rank-preserving and surjective and $f^{-1}\left(\langle q\rangle\right)$ 
is graded, all maximal
elements of $f^{-1}\left(\langle q\rangle\right)$ are mapped to $q$ and one of these has to be greater than $p$. 
Hence, $\bar{f}$ is surjective.
For condition (i), note that for $q\in Q_{\geq f(p)}$ the fiber $\bar{f}^{-1}\left(\langle q\rangle\right)$ equals 
$f^{-1}\left(\langle q\rangle\right)\cap P_{\geq p}$. 
Thus, it is a closed principal upper order ideal of the homotopy 
Cohen-Macaulay poset 
$f^{-1}\left(\langle q\rangle\right)$ and as such homotopy Cohen-Macaulay. 

It remains to verify condition (ii). 
Since $x>p$, we have $f(x)=q_0\in Q_{\geq f(p)}$ and we obtain that $\bar{f}^{-1}(q_0)=\{x\}$. 
In addition, it holds that 
$\bar{f}((p,v)_P)-\{q_0\}=\left(f(I)-\{q_0\}\right)\cap (f(p),f(v))_Q$. Thus, $\bar{f}((p,v)_P)-\{q_0\}$ 
is an open principal upper order ideal of the homotopy Cohen-Macaulay poset $f(I)-\{q_0\}$ and as such homotopy Cohen-Macaulay. 

Now let $q>q_0$ and let $\bar{p}\in \bar{f}^{-1}(q)\cap (p,v)_P$. 
The poset $[p,\bar{p}]_P-\{x\}$ is a closed interval of 
$[u,\bar{p}]_P-\{x\}$. Since by hypothesis the latter one is homotopy Cohen-Macaulay, so is $[p,\bar{p}]_P-\{x\}$.  
Finally, it follows by induction that 
$\widetilde{I}_{>p}=(p,v)_P-\{x\}$ is homotopy Cohen-Macaulay. 
This finishes the first part of the proof. 
The statement concerning double homotopy Cohen-Macaulayness follows directly from the definition of 
this property and the first part of the theorem. 
\qed

\smallskip

\noindent{\bf{Proof of Corollary \ref{th:PFT-2HCM}}.} 
Let $\hat{P}=P\cup\{\hat{0}_P,\hat{1}_P\}$ and 
$\hat{Q}=Q\cup\{\hat{0}_Q,\hat{1}_Q\}$ denote the posets 
obtained from $P$ and $Q$, respectively, by adding a minimum and a maximum element. 
Since $P$ is graded, so is $\hat{P}$.  
Similarly, $\hat{Q}$ is homotopy Cohen-Macaulay, since $Q$ is.  
We consider the map $\hat{f}:\hat{P}\to\hat{Q}$ that extends $f$ by setting $\hat{f}(\hat{0}_P)=\hat{0}_Q$ 
and $\hat{f}(\hat{1}_P)=\hat{1}_Q$. 
It follows from the properties of $f$ that $\hat{f}$ is a surjective rank-preserving poset map, such that for 
$q\in \hat{Q}-\{\hat{1}_Q\}$ the fibers $\hat{f}^{-1}\left(\langle q\rangle\right)$ are homotopy Cohen-Macaulay.  
Theorem \ref{Quillen} further implies that $P$ and thus also the fiber 
$\hat{f}^{-1}\left(\langle\hat{1}_Q\rangle\right)=\hat{P}$ is homotopy Cohen-Macaulay. 
The result follows by applying Theorem \ref{th:PFT-2HCM-b} to the 
posets $\hat{P}$, $\hat{Q}$, the map $\hat{f}$ and the interval $(\hat{0}_P,\hat{1}_P)$.
\qed

\smallskip

It seems natural to ask whether a more general version of Corollary \ref{th:PFT-2HCM} holds for $k$-homotopy Cohen-Macaulay 
posets where $k\geq 2$. Recall that a poset $P$ is called \emph{$k$-homotopy Cohen-Macaulay} if $P$ is homotopy Cohen-Macaulay and 
if for every $A\subseteq P$ with $|A|\leq k-1$ the poset $P-A$ is homotopy Cohen-Macaulay and of the same rank as $P$. 
For $k$-homotopy Cohen-Macaulay posets we obtain the following generalization of Corollary \ref{th:PFT-2HCM}. 

\begin{proposition} \label{th:PFTk-HCM} 
Let $P$ be a graded poset without a minimum and a maximum element and let $\{x_1,\ldots,x_{k-1}\}$ be a 
$(k-1)$-element subset of $P$. 
Assume that for all $A\subseteq \{x_1,\ldots,x_{k-1}\}$ the poset $P-A$ is graded and that $Q$ is a homotopy Cohen-Macaulay poset.
Let further $f:P\to Q$ be a surjective rank-preserving  poset map which satisfies the following conditions: 
\begin{enumerate}
\item [(i)] For every $q\in Q$ the fiber $f^{-1}\left(\langle q\rangle\right)$ is homotopy Cohen-Macaulay. 
\item [(ii)] There exist $q_1,\ldots,q_{k-1}\in Q$ such that 
\begin{itemize}
\item $f^{-1}(q_i)=\{x_i\}$ for all $1\leq i\leq k-1$ and  for all $S\subseteq \{q_1,\ldots,q_{k-1}\}$ the poset 
$Q-S$ is homotopy Cohen-Macaulay, and
\item for all $S\subseteq \{q_1,\ldots,q_{k-1}\}$ and $q\in \bigcap_{v\in S} Q_{>v}$ and $p\in f^{-1}(q)$ the 
poset $\langle p\rangle -f^{-1}(S)$ is homotopy Cohen-Macaulay.
\end{itemize}
\end{enumerate}
Then $P-\{x_1,\ldots,x_{k-1}\}$ is homotopy Cohen-Macaulay as well. 
If for all $A\subseteq P$ with $|A|=k-1$ there exists a map satisfying the above conditions and if $\rank(P-A)=\rank(P)$, 
then $P$ is $k$-homotopy Cohen-Macaulay.
\end{proposition}

We omit the proof of Theorem \ref{th:PFTk-HCM} since it follows exactly the same steps as the one of Theorem \ref{th:PFT-2HCM-b} 
and does not provide any additional insight.

\subsection{A poset fiber theorem for strongly constructible posets}
Strongly constructible posets were introduced in \cite{ca} in order to prove
that the absolute order on the symmetric group $S_n$ is homotopy Cohen-Macaulay.
We first recall the definition of a strongly constructible poset. 

\begin{definition}
A graded poset $P$ of rank $n$ with a minimum element is \emph{strongly constructible} if either
\begin{itemize}
 \item[(i)] $P$ is bounded and pure shellable, or
\item[(ii)] $P$ can be written as a union of two strongly constructible proper ideals $J_1$, $J_2$ of rank $n$ such that the intersection 
$J_1\cap J_2$ is a strongly constructible poset of rank at least $n-1$. 
\end{itemize}
\end{definition}

Strongly constructible and homotopy Cohen-Macaulay posets are related in the following way.

\begin{lemma}\emph{\cite[Corollary 3.3, Proposition 3.6]{ca}} \label{lem:ConstrCM}
Let $P$ be a strongly constructible poset. Then $P$ is homotopy Cohen-Macaulay.
\end{lemma}

Moreover, as the following theorem shows, strongly constructible posets satisfy an analogue of Quillen's poset fiber theorem \cite[Corollary 9.7]{Q}.

\begin{theorem}\label{th:StrgConstrFiberTheorem}
Let $P$ and $Q$ be graded posets. Let further $f:P\rightarrow Q$ be a surjective rank-preserving poset map.
Assume that for every $q\in Q$ the fiber $f^{-1}\left(\langle q\rangle\right)$ is strongly constructible. If $Q$ is strongly constructible, then so is $P$.
\end{theorem}

\begin{proof}
We proceed by induction on the cardinality of $P$. 
If $Q$ is bounded, then $Q=\langle q \rangle$ for some $q\in Q$. In this case, $P=f^{-1}\left(\langle q\rangle\right)$, which 
by hypothesis is strongly constructible. 
Now assume that $Q$ is unbounded and let $\rank(Q)=n$. 
Let $\hat{0}_Q$ be the minimum of $Q$. Since $f$ is rank-preserving, the elements of the fiber $f^{-1}(\hat{0}_Q)$ are the minimal elements of $P$. 
Strong constructibility of $f^{-1}(\hat{0}_Q)$ further implies that $f^{-1}(\hat{0}_Q)$ contains exactly one element, which shows that $P$
has a minimum.
Since $Q$ is strongly constructible, we can write it as $Q=J_1\cup J_2$, 
where $J_1$ and $J_2$ are strongly constructible proper ideals of rank $n$ 
and $J_1\cap J_2$ is strongly constructible of rank at least $n-1$.
Clearly, $P=f^{-1}(Q)=f^{-1}(J_1\cup J_2)=f^{-1}(J_1)\cup f^{-1}(J_2)$. 
Let $f_1,f_2$ and $f_{12}$ denote the restrictions of $f$ to the sets $f^{-1}(J_1)$, $f^{-1}(J_2)$ and $f^{-1}(J_1\cap J_2)$, respectively. 
Each one of these restrictions is a surjective rank-preserving poset map (as $f$ is) and for 
all $q_1\in J_1$, $q_2\in J_2$ and $q_{12}\in J_1\cap J_2$ the fibers 
$f_1^{-1}\left(\langle q_1\rangle\right)$, $f_2^{-1}\left(\langle q_2\rangle\right)$ 
and $f_{12}^{-1}\left(\langle q_{12}\rangle\right)$ are equal 
to $f^{-1}\left(\langle q_1\rangle\right)$, $f^{-1}\left(\langle q_2\rangle\right)$ 
and $f^{-1}\left(\langle q_{12}\rangle\right)$, respectively. 
For this reason they are strongly constructible. 
Thus, it follows by induction that the posets $f^{-1}(J_1)$, $f^{-1}(J_2)$ and $f^{-1}(J_1\cap J_2)=f^{-1}(J_1)\cap f^{-1}(J_2)$ are strongly constructible. 
Since $f$ is a rank-preserving poset map, we infer that $f^{-1}(J_1)$ and $f^{-1}(J_2)$ are order ideals 
of $P$ of rank $n$ and that their intersection is an order ideal of the same rank as $J_1\cap J_2$ which is at least $n-1$. 
\end{proof}

Since the poset of injective words $\I_n$ has been shown to be shellable \cite[Theorem 6.1.(i)]{BW2}, 
it is in particular strongly constructible. Using Theorem \ref{th:StrgConstrFiberTheorem} we can give a direct 
proof of this statement. Moreover this proof will be used in that of Theorem \ref{th:injective2Constructible}.

\begin{example}\label{th:InjectiveWordsConstr}
The poset of injective words $\I_n$ is strongly constructible.
\end{example}

In order to show that $\I_n$ is strongly constructible 
we proceed by induction on $n$. The result is straightforward to verify if 
$n\leq 2$. 
So as to apply Theorem \ref{th:StrgConstrFiberTheorem} we need to define an appropriate map. 
For every $w\in \I_n$, let $\pi(w)$ denote the word obtained from $w$ by 
deleting the letter $n$, if $n\leq w$. Otherwise, we set $\pi(w)=w$. 
\textit{E.g.}, if $n=5$, then $\pi(12534)=1234$ and $\pi(341)=341$. 
Obviously, $\pi(\I_n)=\I_{n-1}$. 
We define the map $f:\I_n\to \I_{n-1}\times\{\hat{0},\hat{1}\}$
by letting
\[f(w)=\left\{\begin{array}{ll}

(\pi(w),\,\hat{0}), & \mbox{if $n\not\leq w$}, \\
(\pi(w),\,\hat{1}), & \mbox{if $n\leq w$}

\end{array}
\right.\]
for $w\in \I_n$.
By definition, $f$ is a rank-preserving map. We show that $f$ is a poset map and surjective. 
Let $u,v\in \I_n$ with $u\leq v$. Suppose first that $n\not\leq v$. Then, we also have $n\not\leq u$, thus $f(u)=(\pi(u),\hat{0})=(u,\hat{0})$ and 
$f(v)=(\pi(v),\hat{0})=(v,\hat{0})$. It follows that $f(u)\leq f(v)$. Suppose now that $n\leq v$. Then, $f(v)=(\pi(v),\hat{1})$ and $f(u)$ is either equal to $(\pi(u),\hat{0})$ or to $(\pi(u), \hat{1})$.  
Since $\pi(u)\leq \pi(v)$ and $\hat{0}< \hat{1}$, in both cases it holds that $f(u)\leq f(v)$. Altogether, this proves that $f$ is a poset map. 
Let $w\in \I_{n-1}$. Then, $f^{-1}\left((w,\hat{0})\right)=\{w\}$ and every word obtained from $w$ by inserting the letter $n$ into some position 
of $w$ lies in $f^{-1}\left((w,\hat{1})\right)$, which means that $f$ is surjective. 
In order to show strong constructibility of the fibers, we will employ the following description of $f^{-1}\left(\langle q\rangle\right)$.

\noindent{\bf{Claim:}}
For every $q\in \I_{n-1}\times \{\hat{0},\hat{1}\}$ we have $f^{-1}\left(\langle q\rangle\right)=\langle f^{-1}(q)\rangle$. 

The claim is obvious if $q=(w,\hat{0})\in \I_{n-1}\times\{\hat{0},\hat{1}\}$. 
Suppose now that $q=(w,\hat{1})$.  
Since $f$ is a poset map, we have $\langle f^{-1}(q)\rangle\subseteq f^{-1}\left(\langle q\rangle\right)$. 
For the reverse inclusion consider any element $u\in f^{-1}\left(\langle q\rangle\right)$. Then, $f(u)\leq q$ and hence $\pi(u)\leq w$. 
If $n\not\leq u$, then $\pi(u)=u$ and therefore $u\leq w \leq w'$ for every $w'\in f^{-1}\left((w,\hat{1})\right)$. 
This implies that $u\in \langle f^{-1}(q)\rangle$. 
If $n\leq u$, then $u$ is obtained from $\pi(u)$ by inserting the letter $n$ in some place. Let $\pi(u)=u_1\cdots u_k$, 
where the letters $u_i$ are distinct elements of $[n-1]$. Without loss of generality we can assume that $u=n\,u_1\cdots u_k$. 
Since $\pi(u)\leq w$, we can find a word $w'\in f^{-1}\left((w,\hat{1})\right)$ such that the letter $n$ directly precedes 
the letter $u_1$ in $w'$. By construction we obtain $u\leq w'$ and thus, $u\in \langle f^{-1}(q)\rangle$. The claim follows. 

Let $q\in \I_{n-1}\times\{\hat{0},\hat{1}\}$. By the above claim, we know that the fiber 
$f^{-1}\left(\langle q\rangle\right)$ is strongly constructible if and only if the order ideal 
$\langle f^{-1}(q)\rangle$ is so.  
If $q=(w,\hat{0})$ for some $w\in \I_{n-1}$, then it holds that 
$\langle f^{-1}(q)\rangle=\langle w\rangle$, \emph{i.e.}, the fiber is a closed interval in $\I_n$. As such 
it is shellable (see Section \ref{subsect:Injective}) and in particular strongly constructible.

Now suppose that $q=(w,\hat{1})$. 
Without loss of generality, we may assume that $w=123\cdots k$, for some $k\leq n-1$. 
Then, $\langle f^{-1}(q)\rangle=\bigcup_{i=0}^{k} \langle 12\cdots i\,n\,i+1\cdots k\rangle$. 
For every $i\in \{0,1,\dots, k\}$, the ideal $S_i:=\langle 12\cdots i\,n\,i+1\cdots k\rangle$ is shellable and 
therefore strongly constructible and $\rank(S_i)=k+1$. 
We show by induction on $j$ that the union $\bigcup_{i=0}^{j} S_i$ is strongly constructible and of rank $k+1$. 
As, by the induction hypothesis, $S_j$ and $\bigcup_{i=0}^{j-1} S_i$ are strongly constructible of rank $k+1$, 
it suffices to show that $S_j\cap \left(\bigcup_{i=0}^{j-1} S_i\right)$ is strongly constructible of rank $k$. 
We have 
\[S_j\cap \left(\bigcup_{i=0}^{j-1} S_i\right)=
\langle  12\cdots k\rangle\cup\langle12\cdots j-1\,n\,j+1\cdots k\rangle.\] 
Both ideals, $\langle  12\cdots k\rangle$ and $\langle12\cdots j-1\,n\,j+1\cdots k\rangle$, are strongly 
constructible of rank $k$ and their intersection is equal to $\langle 12\cdots j-1\, j+1\cdots k\rangle$, 
which is a strongly constructible ideal of rank $k-1$. Therefore, $S_j\cap \left(\bigcup_{i=0}^{j-1} S_i\right)$ is 
strongly constructible of rank $k$ and so is $\bigcup_{i=0}^{j} S_i$, but of rank $k+1$. Conclusively, we have shown that for  
each $q\in \I_{n-1}\times \{\hat{0},\hat{1}\}$ the fiber $ f^{-1}\left(\langle q\rangle\right)$ is strongly constructible. 

By induction, we can assume that $\I_{n-1}$ is strongly constructible and it follows that 
the same is true for the direct product $\I_{n-1}\times\{\hat{0},\hat{1}\}$ (see \cite[Lemma 3.7]{ca}).
We can finally apply Theorem \ref{th:StrgConstrFiberTheorem} and thereby conclude that $\I_n$ is strongly constructible.
\qed

\section{Applications of Theorem \ref{th:PFT-2HCM-b}} \label{sect:Application2}

In this section we give an application of Theorem \ref{th:PFT-2HCM-b} to the lattices of non-crossing partitions of types $A$ and $B$. 
More precisely, we show that the proper part of these lattices is doubly homotopy Cohen-Macaulay. 
For our arguments to work, it will be crucial to reduce to the removal of elements which are fixed point free. 
As soon as this has been achieved, we are able to provide a proof of Theorem \ref{th:NC}, which is case-independent.

For the proofs of Theorem \ref{th:NC} and \ref{th:injective2Constructible} we will need the following technical result.

\begin{theorem}
\label{martina}
Let $P$ be a poset of rank $n$ with a minimum element. Let
$\widetilde{P}=\bar{P}$ if $P$ is bounded, and let $\widetilde{P}=P-\{\hat{0}_P\}$ if $P$ does not have a maximum. 
Assume that $\widetilde{P}$ is doubly homotopy Cohen-Macaulay.
Then, for every $x\in \widetilde{P}$ the poset $(P\times \{\hat{0},\hat{1}\})-\{(x,\hat{0})\}$ is homotopy Cohen-Macaulay of rank $n+1$. 
\end{theorem}

\begin{proof}
Let $x\in\widetilde{P}$ be an element of rank $r$. 
We write $(P\times \{\hat{0},\hat{1}\})-\{(x,\hat{0})\}$ in the following way:
\begin{equation}
 (P\times \{\hat{0},\hat{1}\})-\{(x,\hat{0})\}=\left((P-\{x\})\times \{\hat{0},\hat{1}\}\right)\cup \left((P_{<x}\times\{\hat{0},\hat{1}\})\oplus \{(x,\hat{1})\}\oplus (P_{>x}\times\{\hat{1}\})\right).
\label{eq:decompose}
\end{equation}
The first part of the right-hand side of Equation (\ref{eq:decompose}) accounts for all chains in $(P\times \{\hat{0},\hat{1}\})-\{(x,\hat{0})\}$ not 
containing $(x,\hat{1})$. All chains in $(P\times \{\hat{0},\hat{1}\})-\{(x,\hat{0})\}$ 
passing through $(x,\hat{1})$, are captured by the second part of the right-hand side of 
Equation (\ref{eq:decompose}).
In what follows, we show that those two posets are homotopy Cohen-Macaulay of rank $n+1$ and that so is their intersection of rank $n$.

From the double homotopy Cohen-Macaulayness of $x\in\widetilde{P}$ we infer that $P-\{x\}$ is homotopy Cohen-Macaulay 
of rank $n$. 
Corollary 3.8 in \cite{BWW2} implies that $(P-\{x\})\times \{\hat{0},\hat{1}\}$ is homotopy Cohen-Macaulay of rank $n+1$. 
This takes care of the first poset on the right-hand side of Equation (\ref{eq:decompose}). 

For the second one, note that, since $\tilde{P}$ and thus $P$ are homotopy Cohen-Macaulay, so are $P_{<x}$ and $P_{>x}$ and in particular 
$P_{>x}\times \{\hat{1}\}$. Hence, again by \cite[Corollary 3.8]{BWW2}, 
we deduce that $P_{<x}\times \{\hat{0},\hat{1}\}$ is homotopy Cohen-Macaulay of rank $r$. 
Moreover, since homotopy Cohen-Macaulayness is preserved under taking ordinal sums (see \cite[Corollary 3.4]{BWW2}) 
also $(P_{<x}\times\{\hat{0},\hat{1}\})\oplus \{(x,\hat{1})\}\oplus (P_{>x}\times\{\hat{1}\})$ is homotopy Cohen-Macaulay of rank $r+1+(n-r)=n+1$.

The intersection of the two posets which were considered until now is given as
\begin{equation*}
\left((P-\{x\})\times \{\hat{0},\hat{1}\}\right)\cap\left((P_{<x}\times\{\hat{0},\hat{1}\})\oplus \{(x,\hat{1})\}\oplus (P_{>x}\times\{\hat{1}\})\right)\\
=(P_{<x}\times\{\hat{0},\hat{1}\})\oplus (P_{>x}\times\{\hat{1}\}).
\end{equation*}

$(P_{<x}\times\{\hat{0},\hat{1}\})\oplus (P_{>x}\times\{\hat{1}\})$ is obtained from 
$(P_{<x}\times\{\hat{0},\hat{1}\})\oplus \{(x,\hat{1})\}\oplus (P_{>x}\times\{\hat{1}\})$ by deleting the element $(x,\hat{1})$. 
Combining the facts that rank-selection preserves homotopy Cohen-Macaulayness (see \textit{e.g.}, \cite{bjo1}) and that
$(x,\hat{1})$ is the only element of rank $r+1$ of $(P_{<x}\times\{\hat{0},\hat{1}\})\oplus \{(x,\hat{1})\}\oplus (P_{>x}\times\{\hat{1}\})$,  
we conclude that the intersection $(P_{<x}\times\{\hat{0},\hat{1}\})\oplus (P_{>x}\times\{\hat{1}\})$ is 
homotopy Cohen-Macaulay of rank $n$. 
Eventually, if one applies Lemma 4.9 from \cite{WZZ} to the order complex of 
$(P\times \{\hat{0},\hat{1}\})-\{(x,\hat{0})\}$ as well as to its links, one arrives at the conclusion that $(P\times \{\hat{0},\hat{1}\})-\{(x,\hat{0})\}$ 
is homotopy Cohen-Macaulay of rank $n+1$.
\end{proof}

\smallskip

In order to perform the reduction to the removal of fixed point free permutations, we will use the so-called \emph{Kreweras complement}. 
Given a finite reflection group $W$ and $\mu\in N(W)$,  
the map $K^{\mu}:\NC(W)\to \NC(W)$, which sends $w$ to $K(w)=w^{-1}\mu$, is called the \emph{Kreweras complement} on 
$[e,\mu]$. It was shown in \cite[Lemma 2.5.4]{arm} that this map is an anti-automorphism of the interval $[e,\mu]$, which in particular 
implies that $[e,\mu]$ is self-dual (see Corollary \ref{cor:selfDual}). If $c$ is a Coxeter element of $W$, we write $K$ instead of $K^c$. 
For $W=S_n$ and $W=B_n$, we use the Coxeter elements $c=(1\,2\cdots n)$ and $c=[1,2,\ldots,n]$, respectively.

Our reasoning will employ the following properties of the Kreweras complement $K$.

\begin{lemma}
\label{fixpoint}
Let $w$ be an element in $\NC^A(n)$ or in $\NC^B(n)$. Then:
\begin{itemize}
 \item[(i)] If $\rank(w)<\frac{n}{2}$, then $w$ has at least one fixed point.
\item[(ii)] If $w$ is fixed point free, then its image $K(w)$ has at least one fixed point.
\end{itemize}
\end{lemma}

\begin{proof}
Throughout the proof, we treat $\NC^A(n)$ and $\NC^B(n)$ separately. 

\noindent\emph{Proof of (i).} 
Let $w\in\NC^A(n)$ and let $s$ be the number of cycles in the cycle 
decomposition of $w$. Assume that $n$ is even, \emph{i.e.}, $n=2k$ for 
some positive integer $k$. If $\rank(w)<\frac{n}{2}=k$, then it follows from Section \ref{subsect:NCA} 
that $s\geq n-(k-1)=k+1$. Therefore, $w$ must have at least $k+1$ disjoint cycles in its cycle decomposition. Since 
$2 (k+1)=n+2>n$, we deduce that at least one of those cycles has to be a $1$-cycle, \emph{i.e.}, $w$ has a fixed point. 
The proof for odd $n$ uses the same arguments and is therefore omitted.

We proceed to $\NC^B(n)$. Let $w\in\NC^B(n)$ and let $s$ be the number of paired cycles in the cycle 
decomposition of $w$. Assume that $n$ is even, \emph{i.e.}, $n=2k$ for 
some positive integer $k$. If $\rank(w)<\frac{n}{2}=k$, then it follows from Section \ref{subsect:NCB} 
that $s\geq n-(k-1)=k+1$. This implies that $w$ has at least $k+1$ disjoint paired cycles in its cycle decomposition. Since 
$2(k+1)=n+2>n$, we deduce that at least one of those has to be a paired $1$-cycle, \emph{i.e.}, $w$ has a fixed point. 
The proof for odd $n$ relies on the same reasoning and is therefore left out.

\smallskip

\noindent\emph{Proof of (ii).} 
Let $w\in \NC^A(n)$ be fixed point free. It follows from (i) that we must have $\rank(w)\geq \frac{n}{2}$. 
Since $K$ is an anti-automorphism, we further obtain
\begin{equation*}
\rank(K(w))=(n-1)-\rank(w)\leq
\begin{cases}
n-1-\frac{n}{2}=\frac{n}{2}-1<\frac{n}{2},\quad\mbox{ if } n \mbox{ is even}\\
n-1-\frac{n+1}{2}=\frac{n-3}{2}<\frac{n}{2}, \quad\mbox{if } n \mbox{ is odd.} 
\end{cases}
\end{equation*}
Once more by (i) we infer that $K(w)$ has a fixed point.

It remains to handle the case of $\NC^B(n)$. 
Let $w\in\NC^B(n)$ be an element without a fixed point.
By (i) we know that $\rank(w)\geq\frac{n}{2}$. 
Since $K$ is an anti-automorphism, we further obtain
\begin{equation*}
\rank(K(w))=n-\rank(w)\leq
\begin{cases}
n-\frac{n}{2}=\frac{n}{2},\qquad\qquad\quad\mbox{if } n \mbox{ is even}\\
n-\frac{n+1}{2}=\frac{n-1}{2}<\frac{n}{2}, \quad\mbox{if } n \mbox{ is odd.} 
\end{cases}
\end{equation*}
If $n$ is odd, then (i) implies that $K(w)$ has a fixed point. 
Assume that $n$ is even, \emph{i.e.}, $n=2k$ for some positive integer $k$.
Then $w$ is at least of rank $k$. If $\rank(w)>k$, then the same computation as before shows that 
$\rank(K(w))<k=\frac{n}{2}$ and by (i) this means that $K(w)$ has a fixed point.
Finally, let $\rank(w)=k$. Then, we also have $\rank(K(w))=k$. 
Moreover, there must exist exactly $k$ 
disjoint paired cycles in the cycle decomposition of $w$. Since $w$ is fixed point free, it even follows 
that $w$ is a product of disjoint (paired) transpositions. 
In this case, the Kreweras complement can be computed as $K(w)=wc$.
If, in the cycle decomposition of $w$, there exists a cycle of the form $\lleft a,a+1\rright$ with $n>a>0$, then 
$K(w)(a)=wc(a)=w(a+1)=a$, \emph{i.e.}, $K(w)$ has a fixed point. 
If not, then let $\lleft a,b\rright$ be a transposition occurring in the cycle decomposition of $w$ such that $b>0$, $b>|a|$ 
and such that $b-|a|$ is minimal. 
We need to show that $K(w)$ has at least 
one fixed point. Suppose, by contradiction, that $K(w)$ is fixed point free. 
Since $\rank(K(w))=k$, it follows that $K(w)$ is a product of disjoint paired transpositions. 
$b>|a|$ implies that $b\geq 2$ and $|a|<n$. 
Thus, $K(w)(b-1)=wc(b-1)=w(b)=a$ and $\lleft a,b-1\rright$ has to be a cycle of $K(w)$. 
If $a>0$, we can further conclude that $b-1=K(w)(a)=wc(a)=w(a+1)$. Hence, $\lleft b-1,a+1\rright$ has to be one of the paired 
transpositions in the cycle decomposition of $w$. Since $a>0$ and $a\neq b-1$ by assumption, it holds that $b\geq a+2$, 
\emph{i.e.}, $b-1\geq a+1$. Moreover, we have 
$b-1-|a+1|=b-a-2<b-a$, which contradicts the minimality assumption on $\lleft a,b\rright$. Therefore, $K(w)$ needs to have a fixed point. 
If $a<0$, then similar arguments as in the previous case show that $\lleft b-1,a-1\rright$ occurs in the cycle decomposition of $w$ 
and this again yields a contradiction. 
This finishes the proof.
\end{proof}

Finally, we can proceed to the proof of Theorem \ref{th:NC}.

\smallskip

\noindent{\bf{Proof of Theorem \ref{th:NC}}.} 
For every $n\geq 3$, let $\mathcal{J}_n$ denote the order ideal of $\Abs(S_n)$ or $\Abs(B_n)$, 
which is generated by the Coxeter elements of $S_n$ and $B_n$, respectively. Similarly, let $P_n$ be the lattice of non-crossing 
partitions of type A and B, respectively. 
Let $u\in P_n$ for some $n$ be a permutation of rank $s$. We show by induction on $s$ that open intervals $(e, u)$ 
are doubly homotopy Cohen-Macaulay. 
For $s=2$ the result is trivial. Now let $s\geq 3$. 
Without loss of generality, we can assume that $u(n)\neq n$. 
It follows from \cite[Theorem 1.1]{ABW} and \cite[Proposition 2.6.11]{arm} that $\langle u\rangle$ is shellable, 
hence $(e,u)$ is homotopy Cohen-Macaulay. 
We need to show that for every $x\in (e,u)$
the poset $(e,u)-\{x\}$ is homotopy Cohen-Macaulay of rank $s-2$. 
By Lemma \ref{fixpoint} and using that $K^u$ is an anti-automorphism of $\langle u\rangle$, we may assume that $x$ has a fixed point
We can even presume that $x(n)=n$. 
We consider the following map from \cite[Section 4]{Myrto}. 
For every $w\in \mathcal{J}_n$ 
let $\pi(w)$ be the permutation obtained from $w$ by deleting $n$ from its cycle decomposition.
We define $g:\mathcal{J}_n\rightarrow \mathcal{J}_{n-1}\times\{\hat{0},\hat{1}\}$ 
by letting
\[g(w)=\left\{\begin{array}{ll}

(\pi(w),\,\hat{0}), & \mbox{if $w(n)=n$} \\
(\pi(w),\,\hat{1}), & \mbox{if $w(n)\neq n$}

\end{array}
\right.\]
for $w\in \mathcal{J}_n$.

Our goal is to apply Theorem \ref{th:PFT-2HCM-b} to the map $g$, the interval $(e,u)$ and $x\in (e,u)$. 
In \cite[Section 4]{Myrto} it is shown that $g$ is a surjective rank-preserving poset map, whose fibers are homotopy Cohen-Macaulay. 
We note that $(e,u)-\{x\}$ is graded. We consider the element $q_0=(x,\hat{0})\in g((e,u))$. 
By definition, $g^{-1}(q_0)=\{x\}$. 
Moreover, from $u(n)\neq n$, we derive that the permutation $\pi(u)$ is of rank $s-1$ and 
by induction, the open interval $(e, \pi(u))$ of $\mathcal{J}_{n-1}$ is doubly homotopy Cohen-Macaulay. 
It follows from Theorem \ref{martina} that the poset $[e,\pi(u))\times \{\hat{0},\hat{1}\}-\{q_0\}$ is 
homotopy Cohen-Macaulay. Since $(e,\hat{0})$ is the minimum of this poset, we conclude that 
$g((e,u))-\{q_0\}=[e,\pi(u))\times \{\hat{0},\hat{1}\}-\{(e,\hat{0}),q_0\}$ is homotopy Cohen-Macaulay. 
It remains to verify the second part of condition (ii) of Theorem \ref{th:PFT-2HCM-b}. 
Let $q\in g((e,u))$ such that $q>q_0$ and let $p\in g^{-1}(q)\cap (e,u)$. We need to show that 
$[e,p]-\{x\}$ is homotopy Cohen-Macaulay. 
Since the rank of $p$ is at most $s-1$, the induction hypothesis implies that 
$(e, p)$ is doubly homotopy Cohen-Macaulay. In particular, $[e,p] -\{x\}$ is homotopy Cohen-Macaulay. 
Finally, we can apply \ref{th:PFT-2HCM-b} which yields that $(e,u)-\{x\}$ is homotopy Cohen-Macaulay. 
From $\rank((e,u)-\{x\})=\rank((e,u))$ 
we deduce that $(e,u)$ is doubly homotopy Cohen-Macaulay. 
This finishes the proof since the proper part of each non-crossing partition lattice $P_n$ is isomorphic 
to an interval in $P_{n+1}$ of the form $(e,u)$.
\qed

\smallskip

We want to remark that double homotopy Cohen-Macaulayness of the non-crossing partition lattice $\NC^A(n)$ 
can also be concluded by combining \cite[Theorem 6.3]{H} and \cite[Theorem 3.3]{W}. 
In \cite{H} it is shown that the lattice of non-crossing partitions of type A is supersolvable and 
in \cite{W} it is proved that a supersolvable lattice is doubly homotopy Cohen-Macaulay, 
if and only if, the M\"obius function computed in any of its interval is non-zero. 
However, 
to the best of our knowledge, it is not known whether non-crossing partition lattices of other types are supersolvable.


\section{Applications of Corollary \ref{th:PFT-2HCM}} \label{sect:Application1}

In this section we provide the applications of Corollary \ref{th:PFT-2HCM} 
to the poset of injective words and the complexes of injective words, which were 
discussed in Sections \ref{subsect:Injective} and \ref{subsect:CompInj}, respectively. 
It was shown by Baclawski \cite[Corollary 4.3]{Baclawski} that the proper parts of geometric lattices 
are doubly homotopy Cohen-Macaulay. This in particular implies the following.

\begin{corollary}
\label{cor:Boolean1}
 The proper part of the Boolean algebra $\mathcal{B}_n$ is doubly homotopy Cohen-Macaulay.
\end{corollary}

In the proof of Theorem \ref{th:injective2Constructible}, we will have to distinguish the two cases if the 
element which is removed is maximal or not. 
The following simple lemma takes care of the first case.

\begin{lemma}
\label{lem:maxconst}
Let $P$ be a strongly constructible poset of rank $n$ and let $x\in P$ be a maximal element such that $P-\{x\}$ is graded of rank $n$. 
Then, the poset $P-\{x\}$ is strongly constructible of rank $n$.
\end{lemma}

\begin{proof}
Let $x$ be a maximal element of $P$ such that $P-\{x\}$ is graded of rank $n$. Then, 
$x$ cannot be the unique maximal element of $P$. Since $P$ is strongly constructible and -- by the last argument -- not bounded, 
there are proper ideals of $P$, say $J_1$ and $J_2$, 
which are strongly constructible of rank $n$ and 
their intersection $J_1\cap J_2$ is a strongly constructible ideal of rank at least $n-1$. 
Let $x\in J_1$ and $x\not\in J_2$. The case $x\in J_2$ can be treated similarly. 
Using induction, we may assume that $J_1=\langle x\rangle$. Since  $P-\{x\}$ is graded of rank $n$, 
it follows that every element which is covered by $x$ is also covered by at least one maximal element of $J_2$. 
Thus, $J_1-\{x\}\subseteq J_2$ and therefore $P-\{x\}=(J_1-\{x\})\cup J_2=J_2$, which by assumption is strongly constructible of rank $n$. 
\end{proof}

We can finally give the proof of our fourth main result Theorem \ref{th:injective2Constructible}, \emph{i.e.}, show that the 
propert part of the poset of injective words, $\bar{\I}_n$, is doubly homotopy Cohen-Macaulay. 

\smallskip

\noindent{\bf{Proof of Theorem \ref{th:injective2Constructible}.}}
Throughout the proof, we use $\tilde{\I}_n$ to denote $\I_n-\{\emptyset\}$. 
We proceed by induction on $n$.
If $n=2$, then $\tilde{\I}_2$ has two maximal elements (the words $12$ and $21$) and two elements ($1$ and $2$) 
of rank $1$. No matter which one of the elements $12$, $21$, $1$ or $2$ is removed from $\tilde{\I}_2$, 
the resulting poset is homotopy Cohen-Macaulay of rank $1$. Thus, $\tilde{\I}_2$ is doubly homotopy Cohen-Macaulay.

Now, let $n\geq 3$.
If $x$ is a maximal element of $\tilde{\I}_n$, then, from Example \ref{th:InjectiveWordsConstr} and Lemma \ref{lem:maxconst}, we infer 
that $\I_n-\{x\}$ is strongly constructible and therefore, by Lemma \ref{lem:ConstrCM}, 
$\tilde{\I}_n-\{x\}$ is homotopy Cohen-Macaulay (of rank $n-1$). 
Now, consider an element $x\in \tilde{\I}_n$ 
that is not maximal. Without loss of generality, 
we may assume that $x=12\cdots k$ for some $1\leq k\leq n-1$. 
We consider the restriction of the map $f$ to $\tilde{I}_n$, defined in Example \ref{th:InjectiveWordsConstr}. 
Note that  $f(\tilde{I}_n)= \left(I_{n-1}\times \{\hat{0},\hat{1}\}\right)-\{(\emptyset,\hat{0})\}$. 
Our aim is to apply Corollary \ref{th:PFT-2HCM} to this map. 
We have seen in Example \ref{th:InjectiveWordsConstr} that 
$f$ is a surjective rank-preserving poset map, whose fibers are strongly constructible, hence homotopy Cohen-Macaulay. Those properties 
still hold for the considered restriction. Clearly, $\tilde{\I}_n-\{x\}$ is 
a graded poset. 
Let $q_0=(x,\hat{0})$. Then, $f(x)=q_0$ and $f^{-1}(q_0)=\{x\}$ by definition of $f$. 
By induction, we may assume that $\tilde{\I}_{n-1}$ is doubly homotopy Cohen-Macaulay and it now follows from Theorem \ref{martina} that 
$\left(\I_{n-1}\times\{\hat{0},\hat{1}\}\right)- \{q_0\}$ is homotopy Cohen-Macaulay. 
Since $(\emptyset,\hat{0})$ is the minimum of this 
poset, we conclude that $\left(\I_{n-1}\times\{\hat{0},\hat{1}\}\right)- \{(\emptyset,\hat{0}),q_0\}$ is homotopy Cohen-Macaulay. 
Let $q\in \I_{n-1}\times\{\hat{0},\hat{1}\}-\{(\emptyset,\hat{0}\}$ such that $q> q_0$ and let $p\in f^{-1}(q)$. 
We need to show that the ideal $\langle p\rangle_{\tilde{\I}_n}-\{x\}$ is homotopy Cohen-Macaulay. 
We know from Section \ref{subsect:Injective} that $\langle p\rangle_{\I_n}$ is isomorphic to a Boolean algebra. 
Since $x<p$ (\textit{i.e.}, $x$ is not the maximal
element of $\langle p\rangle_{\I_n}$), we deduce from Corollary \ref{cor:Boolean1} that 
$\langle p\rangle_{\I_n}-\{x\}$ is homotopy Cohen-Macaulay. Hence, so is $\langle p\rangle_{\tilde{\I}_n}-\{x\}$. 
We can finally apply Corollary \ref{th:PFT-2HCM} which yields that $\tilde{\I}_n-\{x\}$ is 
homotopy Cohen-Macaulay. Using that $\rank(\tilde{\I}_n-\{x\})=\rank(\tilde{\I}_n)$ for any 
$x\in \tilde{\I}_n$, we conclude that $\tilde{\I}_n$ is doubly homotopy Cohen-Macaulay. 
\qed

\smallskip

The second application of Corollary \ref{th:PFT-2HCM} we give, is Theorem \ref{th:general} which 
shows that double homotopy Cohen-Macaulayness is preserved when passing from a 
simplicial complex $\Delta$ to a complex of injective words of the form $\Gamma(\Delta,P)$ or $\Gamma/G(\Delta)$.
This result extends Theorem 1.3 in \cite{JW}.

\smallskip

\noindent{\bf{Proof of Theorem \ref{th:general}.}}
We first prove (i). We need to show that for a vertex $v$ of $\Gamma(\Delta,P)$ the 
complex $\Gamma(\Delta,P)-\{v\}$ is homotopy Cohen-Macaulay of the same dimension as $\Gamma(\Delta,P)$. 
Let $f:\Gamma(\Delta,P)-\{\emptyset\}\rightarrow\Delta-\{\emptyset\}$ be the map defined by setting
$f(w_1\cdots w_s)=\{w_1,\ldots ,w_s\}$ for $w=w_1\cdots w_s\in \Gamma(\Delta,P)-\{\emptyset\}$. 
It is shown in \cite[Theorem 1.3]{JW},
that $f$ is a surjective rank-preserving poset map with the property that for a simplex $\sigma\in \Delta$ the fiber 
$f^{-1}\left(\langle \sigma\rangle \right)$ is homotopy Cohen-Macaulay. 
Let $v\in \Gamma(\Delta,P)$ be a vertex. As $v$ is just a single letter in $[n]$, 
we have that $f^{-1}\left(\{v\}\right)=\{v\}$. Since $v$ is also a vertex of $\Delta$ and $\Delta$ is doubly homotopy 
Cohen-Macaulay, we further know that $\Delta-\{v\}$ is homotopy Cohen-Macaulay and pure. Therefore, also $\Gamma(\Delta,P)-\{\emptyset,v\}$ 
is pure. 
It remains to verify the second part of condition (ii) of Corollary \ref{th:PFT-2HCM}. 
Let $\sigma\in \Delta$ with $\{v\}\subsetneq\sigma$ and let 
$\tau\in f^{-1}(\sigma)$. Since $\Gamma(\Delta,P)$ is a 
Boolean cell complex \cite{JW}, the ideal $\langle \tau\rangle$ has to be isomorphic to a Boolean algebra and we deduce from 
Corollary \ref{cor:Boolean1} that $\langle \tau \rangle -\{v\}$ is homotopy Cohen-Macaulay. We can apply Corollary \ref{th:PFT-2HCM} and 
thereby obtain that $\Gamma(\Delta,P)-\{v\}$ is homotopy Cohen-Macaulay. 
The claim now follows since $\dim \Delta=\dim(\Delta-\{v\})$ (by homotopy Cohen-Macaulayness of $\Delta$) and hence, 
$\dim\left(\Gamma(\Delta,P)-\{v\}\right)=\dim \Gamma(\Delta,P)$. 

We now show (ii).  We need to verify that for any vertex $v$ of $\Gamma/G(\Delta)$ the 
complex $\Gamma/G(\Delta)-\{v\}$ is homotopy Cohen-Macaulay of the same dimension as $\Gamma/G(\Delta)$. 
Let $f:\Gamma/G(\Delta)-\{\emptyset\}\rightarrow\Delta$ be the map which sends an equivalence class $[w_1\cdots w_s]$ to 
$f([w_1\cdots w_s])=\{w_1,\ldots,w_s\}\in \Delta$. As in (i), it follows from \cite[Theorem 1.3]{JW} that $f$ is a 
surjective rank-preserving poset map whose fibers $f^{-1}\left(\langle\sigma\rangle\right)$ are homotopy Cohen-Macaulay for 
$\sigma\in \Delta$. 
By the same reasoning as in (i), we deduce that for any vertex $v\in\Gamma/G(\Delta)$ it holds that 
$f^{-1}\left(\{v\}\right)=\{[v]\}$ and that $\Delta-\{v\}$ is homotopy Cohen-Macaulay. 
The last property, in particular implies that $\Delta-\{v\}$ is pure and therefore also $\Gamma/G(\Delta)-\{v\}$. 
In order to apply Corollary \ref{th:PFT-2HCM}, 
it remains to show the second part of condition (ii). For this aim, consider 
$\sigma\in \Delta$ such that $\{v\}\subsetneq \sigma$ and let $[\tau]\in f^{-1}(\sigma)$. 
Since $\Gamma/G(\Delta)$ is a Boolean cell complex \cite[Lemma 1.1]{JW}, we 
know that $\langle [\tau]\rangle$ is isomorphic to a Boolean algebra and Corollary \ref{cor:Boolean1} 
implies that $\langle \tau\rangle-\{[v]\}$ is homotopy Cohen-Macaulay. Corollary \ref{th:PFT-2HCM} yields that $\Gamma/G(\Delta)-\{[v]\}$ is 
homotopy Cohen-Macaulay. 
As in the proof of (i), the condition on the dimension follows from $\dim\Delta=\dim(\Delta-\{v\})$, which holds since $\Delta$ is 
doubly homotopy Cohen-Macaulay. 
\qed

\subsection*{Acknowledgments}
Myrto Kallipoliti was partially supported by Erwin Schr\"odinger International Institute for Mathematical Physics (ESI) 
through a Junior Research Fellowship and by the Austrian
Science Foundation (FWF) through grant Z130-N13. 
Martina Kubitzke was supported by the Austrian Science Foundation (FWF) through grant Y463-N13. 
We would like to thank Christos Athanasiadis for suggesting this problem and Volkmar Welker for 
pointing out his work with Jakob Jonsson about complexes of injective words. 
We are also grateful for their helpful and valuable comments. 
A summary of part of this work will be published in the Proceedings of FPSAC 2011.

\end{document}